\documentclass[11pt]{article}
\usepackage{amsthm}
\usepackage{amsfonts}
\usepackage{amsmath}
\usepackage{amssymb}
\usepackage{colonequals}
\usepackage{epsfig}

\newtheorem{theorem}{Theorem}
\newtheorem{corollary}[theorem]{Corollary}
\newtheorem{lemma}[theorem]{Lemma}

\title{A simplified discharging proof of Gr\"otzsch theorem}
\author{Zden\v{e}k Dvo\v{r}\'ak\thanks{Computer Science Institute, Charles University, Prague, Czech Republic. E-mail: {\tt rakdver@iuuk.mff.cuni.cz}.}}

\begin{document}
\maketitle

\begin{abstract}
In this note, we combine ideas of several previous proofs in order to obtain a quite short proof of Gr\"otzsch theorem.
\end{abstract}

Gr\"otzsch~\cite{Gro} proved that every planar triangle-free graph is $3$-colorable, using the discharging method.
This proof was simplified by Thomassen~\cite{ThoGro} (who also gave a principally different proof~\cite{ThoShortlist}).
Dvo\v{r}\'ak et al.~\cite{DKT} give another variation of the discharging proof.  Both of the later arguments were developed in
order to obtain more general results (the Thomassen's proof gives extensions to girth $5$ graphs in the torus and the projective plane,
while the proof of Dvo\v{r}\'ak et al. aims at algorithmic applications), and thus their presentation of the proof of Gr\"otzsch theorem
is not the simplest possible.  In this note, we provide a streamlined version of the proof, suitable for teaching purposes.

We use the discharging method.  Thus, we consider a hypothetical minimal counterexample to Gr\"otzsch theorem (or more precisely,
its generalization chosen so that we are able to deal with short separating cycles) and show that it does not contain any of several
``reducible'' configurations.  Then, we assign charge to vertices and edges so that the total sum of charges is negative, and
redistribute the charge (under the assumption that no reducible configuration appears in the graph) so that the final charge of each
vertex and face is non-negative.  This gives a contradiction, showing that there exists no counterexample to Gr\"otzsch theorem.

A $3$-coloring $\varphi$ of a cycle $C$ of length at most $6$ is \emph{valid} if either $|C|\le 5$, or $|C|=6$ and there
exist two opposite vertices $u,v\in V(C)$ (i.e., both paths in $C$ between $u$ and $v$ have length three) such that $\varphi(u)\neq\varphi(v)$.
If $G$ is a plane triangle-free graph whose outer face is bounded by an induced cycle $C$ of length at most $6$ and $\varphi$ is
a valid coloring of $C$, then we say that the pair $(G,\varphi)$ is \emph{valid}.  We define a partial ordering $<$ on valid pairs as follows.
We have $(G_1,\varphi_1)<(G_2,\varphi_2)$ if either $|V(G_1)|<|V(G_2)|$, or $|V(G_1)|=|V(G_2)|$ and $|E(G_1)|>|E(G_2)|$.
A valid pair $(G,\varphi)$ is a \emph{minimal counterexample} if $\varphi$ does not extend to a $3$-coloring of $G$, but
for every valid pair $(G',\varphi')<(G,\varphi)$, the coloring $\varphi'$ extends to a $3$-coloring of $G'$.

Let us start with several basic reductions (eliminating short separating cycles, $4$- and $6$-faces), which are mostly standard.  Usually,
$6$-faces are eliminated by collapsing similarly to $4$-faces, which is necessary in the proofs that first eliminate the $4$-cycles and then
maintain girth five; in our setting, adding edges to transform them to $4$-faces is more convenient.

\begin{lemma}\label{lemma-basic}
If $(G,\varphi)$ is a minimal counterexample, then $G$ is $2$-connected, $\delta(G)\ge 2$, all vertices of degree two
are incident with the outer face, and every $(\le\!5)$-cycle in $G$ bounds a face.
\end{lemma}
\begin{proof}
If $G$ contained a vertex $v$ of degree at most two not incident with the outer face, then since $(G,\varphi)$ is
a minimal counterexample, the coloring $\varphi$ extends to a $3$-coloring of $G-v$.  However, we can then color $v$ differently
from its (at most two) neighbors, obtaining a $3$-coloring of $G$ extending $\varphi$.  This is a contradiction, and thus $G$ contains
no such vertex.  Note that all vertices of $G$ incident with the outer face have degree at least two, since the outer face is bounded
by a cycle.

Suppose that a $(\le\!5)$-cycle $K$ of $G$ does not bound a face.  Since $G$ is triangle-free, the cycle $K$ is induced.
Let $G_1$ be the subgraph of $G$ drawn outside (and including) $K$, and let $G_2$ be the subgraph of $G$ drawn inside (and including) $K$.
We have $(G_1,\varphi)<(G,\varphi)$, and thus there exists a $3$-coloring $\psi_1$ of $G_1$ extending $\varphi$.
Furthermore, $(G_2,\psi_1\restriction V(K))<(G,\varphi)$, and thus there exists a $3$-coloring $\psi_2$ of $G_2$ that matches $\psi_1$ on $K$.
The union of $\psi_1$ and $\psi_2$ is a $3$-coloring of $G$ extending $\varphi$, which is a contradiction.
Hence, every $(\le\!5)$-cycle of $G$ bounds a face.

Suppose that $G$ is not $2$-connected, and thus there exist graphs $G_1$, $G_2$ intersecting in at most one vertex
such that $G=G_1\cup G_2$, $C\subseteq G_1$ and $|V(G_1)|,|V(G_2)|\ge 4$.  Observe that for $i\in\{1,2\}$, there
exists a vertex $v_i\in V(G_i)$ incident with the common face of $G_1$ and $G_2$ such that if $G_1$ and $G_2$ intersect,
then the distance between $v_i$ and the vertex in $G_1\cap G_2$ is at least two.  Then $G+v_1v_2$ is triangle-free
and $(G+v_1v_2,\varphi)<(G,\varphi)$.  However, this implies that there exists a $3$-coloring of $G+v_1v_2$ extending $\varphi$,
which also gives such a $3$-coloring of $G$.  This is a contradiction.
\end{proof}

\begin{lemma}\label{lemma-6c}
If $(G,\varphi)$ is a minimal counterexample with the outer face bounded by a cycle $C$, then $G$ contains no induced $6$-cycle other than $C$.
\end{lemma}
\begin{proof}
Suppose that $G$ contains an induced $6$-cycle $K\neq C$.  
Let $G_1$ be the subgraph of $G$ drawn outside (and including) $K$, and
let $G_2$ be the subgraph of $G$ drawn inside (and including) $K$.  
Since $K\neq C$ and $C$ is an induced cycle, we have $V(K)\not\subseteq V(C)$.
Let us label the vertices of $K$ by $v_1$, $v_2$, \ldots $v_6$ in order so that $v_1\not\in V(C)$ and subject to that, the degree of $v_1$
in $G_1$ is as small as possible. 
Let $G'_1=G_1+v_1v_4$.

Note that $C$ is an induced cycle bounding the outer face of $G'_1$.  If $G'_1$ contains a triangle, then $G$ contains
a $5$-cycle $Q=v_1v_2v_3v_4x$ with $x\in V(G_1)\setminus V(K)$,
which bounds a face by Lemma~\ref{lemma-basic}.  Hence, the path $v_1v_2v_3$ is contained in boundaries of two
distinct faces ($K$ and $Q$) in $G_1$, and thus $v_2$ has degree two in $G_1$.  However, $v_1$ has at least three neighbors $v_2$, $v_3$
and $x$ in $G_1$, which contradicts the choice of the labels of the vertices of $K$.  Therefore, $G'_1$ is triangle-free.
Note also that either $|V(G'_1)|<|V(G)|$ (if $K$ does not bound a face), or $|V(G'_1)|=|V(G)|$ and $|E(G'_1)|>|E(G)|$ (if $K$ bounds a face).
Hence, $(G'_1,\varphi)<(G,\varphi)$, and thus there exists a $3$-coloring $\psi_1$ of $G'_1$ extending $\varphi$.
Because of the edge $v_1v_4$, $\psi_1\restriction V(K)$ is a valid coloring of $K$.
Since $K$ is an induced cycle, we have $V(C)\not\subseteq V(K)$, and thus $|V(G_2)|<|V(G)|$ and $(G_2,\psi_1\restriction V(K))<(G,\varphi)$.
Therefore, there exists a $3$-coloring $\psi_2$ of $G_2$ that matches $\psi_1$ on $K$.
The union of $\psi_1$ and $\psi_2$ is a $3$-coloring of $G$ extending $\varphi$, which is a contradiction.
\end{proof}

\begin{lemma}\label{lemma-4c}
If $(G,\varphi)$ is a minimal counterexample with the outer face bounded by a cycle $C$, then $G$ contains no $4$-cycle other than $C$.
\end{lemma}
\begin{proof}
Suppose that $G$ contains a $4$-cycle $K\neq C$.  By Lemma~\ref{lemma-basic}, $K$ bounds a face.  Let $v_1$, \ldots, $v_4$ be the
vertices of $K$ in order.  Since $K\neq C$ and $C$ is an induced cycle, we can assume that $v_3\not\in V(C)$.  Let $G_1$ be the graph
obtained from $G$ by identifying $v_1$ with $v_3$.  Note that each $3$-coloring of $G_1$ corresponds to a $3$-coloring of $G$, and thus
$\varphi$ does not extend to a $3$-coloring of $G_1$.  Since $|V(G_1)|<|V(G)|$, it follows that the pair $(G_1,\varphi)$ is not valid.
There are two possibilities: either $G_1$ contains a triangle or its outer face is not an induced cycle.

If $G_1$ contains a triangle, then $G$ contains a $5$-cycle $Q=v_1v_2v_3xy$.  By Lemma~\ref{lemma-basic}, $Q$ bounds a face, hence
the path $v_1v_2v_3$ is contained in boundaries of two distinct faces ($K$ and $Q$).  It follows that
$v_2$ has degree two, and by Lemma~\ref{lemma-basic}, $v_2$ is incident with the outer face.  However, this implies that $v_3$ is incident
with the outer face as well, contrary to its choice.

It remains to consider the case that the outer face of $G_1$ is not an induced cycle.  Since $G_1$ contains no triangle, it follows
that the outer face of $G_1$ has length $6$.  Hence, $C=v_1w_2w_3w_4w_5w_6$ and $v_3$ is adjacent to $w_4$.  We choose the labels so that
either $v_2=w_2$ or $v_2$ is contained inside the $6$-cycle $Q=v_1v_4v_3w_4w_3w_2$.  By Lemma~\ref{lemma-6c}, $Q$ is not an induced cycle,
and since $C$ is an induced cycle and $G$ is triangle-free, we conclude that $v_3w_2\in E(G)$.  The symmetric argument for the $6$-cycle
$v_1v_2v_3w_4w_5w_6$ implies that $v_3w_6\in E(G)$.
By Lemma~\ref{lemma-basic}, $w_2v_1w_6v_3$, $w_2v_3w_4w_3$ and $w_6v_3w_4w_5$ bound faces,
hence $V(G)=V(C)\cup \{v_3\}$.  Since $\varphi$ is a valid coloring of $C$, two opposite vertices of $C$ have different colors;
say $\varphi(v_1)\neq \varphi(w_4)$.  Then, we can properly color $v_3$ by $\varphi(v_1)$.  This is a contradiction.
\end{proof}

\begin{corollary}\label{cor-6c}
If $(G,\varphi)$ is a minimal counterexample with the outer face bounded by a cycle $C$, then $G$ contains no $6$-cycle other than $C$.
\end{corollary}
\begin{proof}
No $6$-cycle in $G$ other than $C$ is induced by Lemma~\ref{lemma-6c}.  However, a non-induced $6$-cycle would imply the presence
of at least two $4$-cycles, contradicting Lemma~\ref{lemma-4c}.
\end{proof}

The following is the main reduction enabling us to eliminate $5$-faces incident with too many vertices of degree three.
Thomassen~\cite{ThoGro} uses a different reduction in this case, which however is slightly more difficult to argue about.

\begin{lemma}\label{lemma-mainred}
Let $(G,\varphi)$ be a minimal counterexample whose outer face is bounded by a cycle $C$.  Let $K=v_1v_2v_3v_4v_5$ be a cycle bounding
a $5$-face in $G$ such that $v_1$, $v_2$, $v_3$ and $v_4$ have degree three and do not belong to $V(C)$.  Then at least one
of the neighbors of $v_1$, \ldots, $v_4$ outside $K$ belongs to $V(C)$.
\end{lemma}
\begin{proof}
Let $x_1$, \ldots, $x_4$ be the neighbors of $v_1$, \ldots, $v_4$, respectively, outside of $K$.  Suppose that none of these vertices
belongs to $V(C)$.  Let $G'$ be the graph obtained from $G-\{v_1,v_2,v_3,v_4\}$ by adding the edge $x_1x_4$ and by identifying $x_2$ with $x_3$.
Note that $C$ is an induced cycle bounding the outer face of $G'$.

If $G'$ contained a triangle, then $G$ would contain
a $6$-cycle $x_2v_2v_3x_4ab$ or $x_1v_1v_5v_4x_4a$ (contrary to Corollary~\ref{cor-6c}) or a matching between $\{x_1,x_4\}$ and $\{x_2,x_3\}$
(contrary to either planarity or Lemma~\ref{lemma-4c}).  Hence, $(G',\varphi)<(G,\varphi)$ is valid and there exists a $3$-coloring $\psi$ of $G'$
extending $\varphi$.  Note that $\psi(x_1)\neq\psi(x_4)$; hence, we can choose colors $\psi(v_1)\not\in \{\psi(x_1),\psi(v_5)\}$ and
$\psi(v_4)\not\in \{\psi(x_4),\psi(v_5)\}$ so that $\psi(v_1)\neq \psi(v_4)$.  Since $\psi(x_2)=\psi(x_3)$, observe that we can extend this
coloring to $v_2$ and $v_3$.  This gives a $3$-coloring of $G$ extending $\varphi$, which is a contradiction.
\end{proof}

We can now proceed with the discharging phase of the proof.

\begin{lemma}\label{lemma-disch}
If $(G,\varphi)$ is a valid pair, then $\varphi$ extends to a $3$-coloring of $G$.
\end{lemma}
\begin{proof}
Suppose that $\varphi$ does not extend to a $3$-coloring of $G$; choose a valid pair $(G,\varphi)$ with this property
that is minimal with respect to $<$.  Thus, $(G,\varphi)$ is a minimal counterexample.  Clearly, $G$ has a vertex not incident with
its outer face.
Let the \emph{initial charge} $c_0(v)$ of a vertex $v$ of $G$ be defined as $\deg(v)-4$ and the initial charge $c_0(f)$ of a face $f$ of $G$
as $|f|-4$.

Let $C$ be the cycle bounding the outer face of $G$.  A $5$-face $Q$ is \emph{tied} to a vertex $z\in V(C)$ if
$z\not\in V(Q)$ and $z$ has a neighbor in $V(Q)\setminus V(C)$ of degree three.
Let us redistribute the charge as follows: each face other than the outer one sends $1/3$ to each incident vertex that
either has degree two, or has degree three and does not belong to $V(C)$.  Each vertex of $C$ sends $1/3$ to each $5$-face tied to
it.  Let the charge obtained by these rules be called \emph{final} and denoted by $c$.

First, let us argue that the final charge of each vertex $v\in V(G)\setminus V(C)$ is non-negative: by Lemma~\ref{lemma-basic},
$v$ has degree at least three.  If $v$ has degree at least four, then $c(v)\ge c_0(v)=\deg(v)-4\ge 0$.  If $v$ has degree three,
then it receives $1/3$ from each incident face, and $c(v)=c_0(v)+1=0$.

Next, consider the charge of a face $f$ distinct from the outer one.  By Lemma~\ref{lemma-4c}, we have $|f|\ge 5$.  The face $f$
sends at most $1/3$ to each incident vertex, and thus its final charge is $c(f)\ge c_0(f)-|f|/3=2|f|/3-4$.  Hence, $c(f)\ge 0$ unless $|f|=5$.
Suppose that $|f|=5$ and let $k$ be the number of vertices to that $f$ sends charge.  We have $c(f)=c_0(f)-k/3=1-k/3$.  If $k\le 3$, then $c(f)\ge 0$,
and thus we can assume that $k\ge 4$.
If $f$ is incident with a vertex $v$ of degree two, then note that $v\in V(C)$ by Lemma~\ref{lemma-basic}.  Furthermore, since
$G$ is $2$-connected and $G\neq C$, we conclude that $f$ is incident with at least two vertices of degree three belonging to $V(C)$, to which
$f$ does not send charge.  This contradicts the assumption that $k\ge 4$.
Hence, no vertex of degree two is incident with $f$, and thus $k$ is the number of vertices of $V(f)\setminus V(C)$ of degree three.
By Lemma~\ref{lemma-mainred}, $f$ is tied to at least $k-3$ vertices of $C$, and thus $c(f)\ge c_0(f)-k/3+(k-3)/3=0$.

The final charge of the outer face is $|C|-4$.  Consider a vertex $v\in V(C)$.  If $\deg(v)=2$, then $v$ receives $1/3$ from the incident non-outer face
and $c(v)=-5/3$.  If $\deg(v)\ge 3$, then $v$ sends $1/3$ to at most $\deg(v)-2$ faces tied to it, and thus
$c(v)\ge c_0(v)-(\deg(v)-2)/3=2\deg(v)/3 - 10/3\ge -4/3$.

Note that since $G$ is $2$-connected and $G\neq C$, the outer face is incident with at least two vertices of degree greater than two.
Therefore, the sum of the final charges is at least $(|C|-4)-5(|C|-2)/3-2\cdot 4/3=-10/3-2|C|/3>-8$, since $|C|\le 6$.
On the other hand, the sum of final charges is equal to the sum of the initial charges, which (if $G$ has $n$ vertices, $m$ edges and $s$ faces) is
\begin{eqnarray*}
\sum_v c_0(v)+\sum_f c_0(f)&=&\sum_v (\deg(v)-4)+\sum_f (|f|-4)\\
&=&(2m-4n)+(2m-4s)=4(m-n-s)\\
&=&-8
\end{eqnarray*}
by Euler's formula.  This is a contradiction.
\end{proof}

The proof of Gr\"otzsch theorem is now straightforward.

\begin{theorem}
Every planar triangle-free graph is $3$-colorable.
\end{theorem}
\begin{proof}
Suppose for a contradiction that $G$ is a planar triangle-free graph that is not $3$-colorable, chosen with as few vertices as possible.
Clearly, $G$ has minimum degree at least three (as otherwise we can remove a vertex $v$ of degree at most two, $3$-color the rest of the
graph by the minimality of $G$, and color $v$ differently from its neighbors).  Hence, Euler's formula implies that every drawing of $G$
in the plane has a face of length at most $5$.  Fix a drawing of $G$ such that the outer face is bounded by a cycle $C$ of length at most
$5$.  Since $G$ is triangle-free, the cycle $C$ is induced.  Let $\varphi$ be an arbitrary $3$-coloring of $C$.  By Lemma~\ref{lemma-disch},
$\varphi$ extends to a $3$-coloring of $G$, which is a contradiction.
\end{proof}

\end{document}